\documentclass[11pt]{amsart}
\usepackage{amsfonts}
\usepackage{graphicx,color}
\usepackage{amsthm}
\usepackage{amssymb,amsmath}

\title[LP bounds for spherical $(k,k)$-designs]{\bf Linear programming bounds for spherical $(k,k)$-designs}

\def\ds{\displaystyle}
\date{\today}

\newtheorem{theorem}{Theorem}[section]

\theoremstyle{definition}
\newtheorem{definition}[theorem]{Definition}
\newtheorem{example}[theorem]{Example}
\newtheorem{remark}[theorem]{Remark}

\author[P. Boyvalenkov]{P. G. Boyvalenkov$^\dagger$}
\address{Institute of Mathematics and Informatics, Bulgarian Academy of Sciences,
8 G Bonchev Str.,
1113  Sofia, Bulgaria \\
and Technical Faculty, South-Western University, Blagoevgrad, Bulgaria.
}
\email{peter@math.bas.bg}
\thanks{\noindent $^\dagger$ This research was partially supported by Bulgarian NSF under project KP-06-N32/2-2019. }

\date{}
\def\ds{\displaystyle}

\begin{document}
\maketitle

\begin{abstract}
We derive general linear programming bounds for spherical $(k,k)$-designs. This includes lower bounds
for the minimum cardinality and lower and upper bounds for minimum and maximum energy, respectively.
As applications we obtain a universal bound in sense of Levenshtein for the minimum possible cardinality 
of a $(k,k)$ design for fixed dimension and $k$ and corresponding optimality result. We also discuss examples and
possibilities for attaining the universal bound. 
\end{abstract}

\section{Introduction}

A nonempty finite set $C \subset \mathbb{S}^{n-1}$ is called a {\em spherical code}. 
The geometry of spherical codes is related to the properties of the Gegenbauer polynomials \cite{Sze}; we consider their 
normalized version $\{P_i^{(n)}(t)\}_{i=0}^\infty$ satisfying the following three-term recurrence relation 
\[ (i+n-2)P_{i+1}^{(n)}(t)=(2i+n-2)tP_i^{(n)}(t)-iP_{i-1}^{(n)}(t), \]
$i=1,2,\ldots$, with initial conditions $P_0^{(n)}(t)=1$ and $P_1^{(n)}(t)=t$. 

Given a code $C \subset \mathbb{S}^{n-1}$, the quantities 
 \begin{equation}\label{Mk0}
 M_i(C):=\sum_{x,y\in C} P^{(n)}_i(\langle x , y \rangle )=|C|+\sum_{x,y\in C, x \neq y} P^{(n)}_i(\langle x , y \rangle ), \ i \geq 1 \end{equation}
are called {\it moments} of $C$. Here $\langle x,y \rangle $ is the usual inner product of $x,y \in \mathbb{S}^{n-1}$. 

The well known positive definiteness of the Gegenbauer polynomials \cite{Sch1942} implies that $M_i(C) \geq 0$ for every $i \geq 1$.
The case of equality (for some indices $i$) is quite important. The concept of spherical $T$-designs was introduced by Delsarte and Seidel \cite{DS89} in 1989. 

\begin{definition} \label{T-designs} \cite{BBTZ17} Let $T$ be a finite set of positive integers. A spherical code $C \subset \mathbb{S}^{n-1}$ 
is called a spherical $T$-design if $M_i(C)=0 \mbox{ for all } i \in T$. 
\end{definition}

The classical case $T=\{1,2,\ldots,m\}$ leads to the spherical $m$-designs introduced by Delsarte, Goethals and Seidel \cite{DGS} in 1977
(see also \cite{Lev-chapter}). The case of $T$ consisting of even integers was considered by Bannai et al in \cite[Section 6.1]{BBTZ17} 
(see also \cite{BOT,DS89,ZBBKY17}). In this paper we consider $T$ consisting of several consecutive even integers $2,4,\ldots$ 
(see \cite{DS89,KP10,Wal-book}). 

\begin{definition} Let $k$ be a positive integer. The set $C \subset \mathbb{S}^{n-1}$ is called a spherical $(k,k)$-design if 
$M_{2i}(C)=0$ for every $i=1,2,\ldots,k$.
\end{definition}

%\begin{definition} Let $k$ be a positive integer. The set $C \subset \mathbb{S}^{n-1}$ is called a $(k,k)$-design if 
%\begin{equation} \label{def1}
%1+\sum_{y \in C \setminus \{x\}} P_{2i}^{(n)}(\langle x,y \rangle) = 0
%\end{equation}
%holds true for every $x \in C$ and every $i=1,2,\ldots,k$. Clearly, \eqref{def1} can be written as
%\begin{equation} \label{def2}
%1+\sum_{i=1}^n A_i(x) P_{2i}^{(n)}\left(1-\frac{2i}{n}\right) = 0,
%\end{equation}
%where $(A_0(x),A_1(x),\ldots,A_n(x))$ is the distance distribution of $C$ with respect to $x$. 
%\end{definition}

It seems that spherical $(k,k)$ designs were first considered in \cite{KP10} (called semi-designs there). 
Recently, theory was developed (see \cite{Wal-book} and references therein) and relations to tight frames (i.e., $(1,1)$-designs) 
were investigated. However, up to best of our knowledge, linear programming for spherical $(k,k)$-designs is not developed yet. 

In Section 2 we formulate three main problems that can be attacked by linear programming. General linear programming 
bounds are derived in Section 3. Section 4 is devoted to universal lower bound for the minimum possible cardinality of 
$(k,k)$ designs for fixed $n$ and $k$ and its optimality. In Section 5 we show some examples and classification results
for codes attaining the universal bound. 

\section{Cardinality and energy problems for spherical $(k,k)$-designs}

Designs are, in general sense, good approximations of the space they live. Thus it natural to know designs with as less as 
possible points. Thus, we are interested in the quantity
\[ \mathcal{M}(n,k):=\min \{ |C|: C \subset \mathbb{S}^{n-1} \mbox{ is a $(k,k)$-design}\}, \]
the minimum possible cardinality of a $(k,k)$-design in $\mathbb{S}^{n-1}$. 

Recently, importance of energy of spherical designs was recognized as interesting (see \cite{BDHSS15,GS19} and references therein).
The spherical designs appear to be energy effective; i.e. the upper and lower bounds for their energy are often close each other. Thus it 
is natural to consider energy problems for spherical $(k,k)$-designs. 

\begin{definition}
Given a (potential) function $h(t):[-1,1] \to [0,+\infty]$ and a code $C \subset \mathbb{S}^{n-1}$, the {\em $h$-energy} of $C$ is
\[ E_h(C):=\sum_{x, y \in C, x \neq y} h(\langle x,y \rangle). \]
\end{definition}

Therefore, we are also interested in the minimum and maximum possible $h$-energy of a $(k,k)$-design in $\mathbb{S}^{n-1}$ with given cardinality; i.e., in the quantities
\[ \mathcal{L}_h(n,k,M):=\min \{E_h(C): C \in \mathbb{S}^{n-1} \mbox{ is a $(k,k)$-design}, |C|=M \}, \]
and 
\[ \mathcal{U}_h(n,k,M):=\max \{E_h(C): C \in \mathbb{S}^{n-1} \mbox{ is a $(k,k)$-design}, |C|=M \}. \]

We will introduce general linear programming framework for bounding for the quantities $\mathcal{M}(n,k)$, $\mathcal{L}_h(n,k,M)$,
and $\mathcal{U}_h(n,k,M)$. Then we will derive a universal (in sense of Levenshtein) bound for $\mathcal{M}(n,k)$
as our derivation allows investigations of the optimality of the bounds and the designs which (if exist) would attain these bounds. 

Universal bounds for the energy quantities $\mathcal{L}_h(n,k,M)$ and $\mathcal{U}_h(n,k,M)$ will be considered elsewhere.

\section{General linear programming bounds}

For any real polynomial $f(t)$ we consider its Gegenbauer expansion
\[ f(t)=\sum_{i=0}^m f_i P_i^{(n)}(t), \]
where $m=\deg(f)$, and define the following sets of polynomials 
\[ F_{n,k}:=\{ f(t) \, : \, f_0>0, f_i \leq 0, i=1,3,\ldots,2k-1 \mbox{ and } i \geq 2k+1\}, \]
\[ G_{n,k}:=\{ f(t) \, : \, f_0>0, f_i \geq 0, i=1,3,\ldots,2k-1 \mbox{ and } i \geq 2k+1\}. \]
Since any Gegenbauer polynomial $P_j^{(n)}(t)$ is an odd/even function for odd/even $j$, 
any polynomial $f(t)$  which is an even function has $f_i=0$ for its Gegenbauer coefficients with odd $i$.
This yields that if $\deg(f) \leq 2k$, then $f$ belongs to both $F_{n,k}$ and $G_{n,k}$. 

Further, we define
\[ M_{n,k}:=\{ f(t) \in F_{n,k}\, : f(t) \geq 0 \ \forall \, t \in [-1,1] \}, \]
\[ L_{n,k}^{(h)}:=\{ f(t) \in G_{n,k}\, : f(t) \leq h(t) \ \forall \, t \in [-1,1] \}, \]
\[ U_{n,k}^{(h)}:=\{ f(t) \in F_{n,k}\, : f(t) \geq h(t) \ \forall \, t \in [-1,1] \}. \]

Linear programming for spherical designs was introduced by Delsarte, Goethals and Seidel \cite{DGS} and 
developed for energy bounds by Yudin \cite{Y}. All three bounds in Theorem \ref{thm_lp} below follow easily from the identity
\begin{equation}
  \label{main}
  |C|f(1)+\sum_{x,y\in C, x \neq y} f(\langle x,y\rangle) = |C|^2f_0 + \sum_{i=1}^m f_i M_i 
\end{equation}
 (see, for example, \cite[Equation (1.20)]{lev92}, \cite[Equation (3)]{ZBBKY17}), which
serves as a key source of estimations by linear programming. 
It follows easily by computing in two ways the sum $\sum_{x,y\in C} f(\langle x,y\rangle)$ and using the definition of the moments.

We are now in a position to formulate the general linear programming theorems for the
quantities $\mathcal{M}(n,k)$, $\mathcal{L}_h(n,k,M)$, and $\mathcal{U}_h(n,k,M)$.

\begin{theorem} \label{thm_lp}
a) If $n \geq 2$ and $k$ are positive integers and $f \in M_{n,k}$, then $\mathcal{M}(n,k) \geq f(1)/f_0$. 

b) If $n \geq 2$, $k$, and $M \geq 2$ are positive integers, $h$ is a potential function, and $f \in L_{n,k}^{(h)}$, then  
$\mathcal{L}_h(n,k,M) \geq M(f_0M-f(1))$. 

c) If $n \geq 2$, $k$, and $M \geq 2$ are positive integers, $h$ is a potential function, and $f \in U_{n,k}^{(h)}$, then  
$\mathcal{U}_h(n,k,M) \leq M(f_0M-f(1))$. 
\end{theorem}

\begin{proof}
a) Let $C \subset \mathbb{S}^{n-1}$ be a $(k,k)$-design and $f \in M_{n,k}$. We apply \eqref{main} for $C$ and $f$. 
Since  $M_i \geq 0$ for all $i$ and, in particular, $M_{2i}(C)=0$ for $i=2,4,\ldots,2k$, and $f_i \leq 0$ for all odd $i$ and for all even
$i>2k$, the right hand side of \eqref{main} does not exceed $f_0|C|^2$. 
The sum in the left hand side is nonnegative because $f(t) \geq 0$ for every $t \in [-1,1]$. Thus the left hand side
is at most $f(1)|C|$ and we conclude that $|C| \geq f(1)/f_0$. Since this inequality follows for every $C$, we have 
$\mathcal{M}(n,k) \geq f(1)/f_0.$

b) Now let $C  \subset \mathbb{S}^{n-1}$ be a $(k,k)$-design of cardinality $M$ and $f \in L^{(h)}_{n,k}$. We rewrite the left hand side of \eqref{main} for $C$ and $f$ 
\begin{equation} \label{re-main-energy}
f(1)|C|+E_h(C)+\sum_{x,y\in C, x \neq y} \left( f(\langle x,y\rangle)-h(\langle x,y\rangle)\right)=|C|^2f_0 + \sum_{i=1}^m f_i M_i 
\end{equation}
to involve the energy $E_h(C)$.

Similarly to a), we conclude that the right hand side of \eqref{re-main-energy} is at least $f_0|C|^2$ and the 
left hand side does not exceed $f(1)|C|+E_h(C)$ (observe that the sum in the left hand side is nonpositive because of the 
condition $f(t) \leq h(t)$ for every $t \in [-1,1]$). 
Therefore $E_h(C) \geq |C|(f_0|C|-f(1))$. Since this follows for 
every such $C$, we conclude that $\mathcal{L}_h(n,M,k) \geq M(f_0M-f(1))$.

c) If $C  \subset \mathbb{S}^{n-1}$ is a $(k,k)$-design of cardinality $M$ and $f \in U^{(h)}_{n,k}$, then as in b) 
we use \eqref{re-main-energy} to see that $E_h(C) \leq |C|(f_0|C|-f(1))$, whence  
$\mathcal{U}_h(n,M,k) \leq M(f_0M-f(1))$.
\end{proof}

The conditions for achieving equality in all three bounds of Theorem \ref{thm_lp} are obviously the same -- one need to have
inner products $\langle x,y \rangle$, $x,y \in C$, $x \neq y$, only equal to roots of $f(t)$, and $f_iM_i=0$ for all odd $i$ 
and all $i \geq 2k+1$. 

We conclude this section with an application of the addition formula (see \cite[Theorem 3.3]{DGS}, \cite[Section 3]{Lev-chapter} in the designs' context)
\[ P_i^{(n)} (\langle x,y \rangle) = \frac{1}{r_i} \sum_{j=1}^{r_i} v_{ij}(x) v_{ij}(y) \]
where $r_i=\dim \mbox{Harm}(i)$ and $\{v_{ij}(x): j=1,2,\ldots,r_i\}$ is an orthonormal basis of Harm$(i)$,
the space of homogeneous harmonic polynomials of degree $i$ on $\mathbb{S}^{n-1}$. 

\begin{theorem} \label{moments-degs}
We have $M_i(C)=0$ if and only if $\sum_{x \in C} P_i^{(n)}(\langle x,y \rangle)=0$ for any fixed $y \in C$. 
\end{theorem}

\begin{proof}
Computing $M_i(C)$ by the addition formula, we see that $M_i(C)=0$ if and only if $\sum_{x \in C} v(x)=0$ for 
each $v \in \mbox{Harm}(i)$. Using this and the addition formula again we obtain that the double sum in 
\eqref{Mk0} splits into $|C|$ sums each one equal to $0$. Indeed, for fixed $y \in C$, we consecutively obtain
\[ \sum_{x \in C} P_i^{(n)}(\langle x,y \rangle)=\sum_{x \in C} \frac{1}{r_i} \sum_{j=1}^{r_i} v_{ij}(x)\overline{v_{ij}(y)} =
 \frac{1}{r_i}  \sum_{j=1}^{r_i} \overline{v_{ij}(y)} \sum_{x \in C}v_{ij}(x) = 0, \]
which completes the proof. \end{proof}

\section{A universal bound for $\mathcal{M}(n,k)$}

Suitable polynomials in Theorem \ref{thm_lp} may give universal (in sense of Levenshtein \cite{Lev-chapter}) bounds. We 
present here such a bound for $\mathcal{M}(n,k)$ using a polynomial which is suggested from the choice of 
Delsarte, Goethals and Seidel in \cite{DGS}.

Denote
 $B(n,m):=\min\{|C|: C \subset \mathbb{S}^{n-1} \mbox{ is a spherical $m$-design}\}$.
The Delsarte-Goethals-Seidel bound \cite{DGS} 
\begin{equation}
\label{DGS-bound}
B(n,m) \geq D(n,m):= \left\{ \begin{array}{ll}
 \ds 2\binom{n+k-2}{k-1}, & \mbox{ if $m=2k-1$,} \\[12pt]
 \ds \binom{n+k-1}{k}+\binom{n+k-2}{k-1}, & \mbox{ if  $m=2k$}.
\end{array}
  \right.
\end{equation}
was obtained by linear programming via the polynomials
\begin{eqnarray}
\label{DGS-poly}
     d_{m}(t) = \left\{
     \begin{array}{ll}
        (t+1)\left(P_{k-1}^{1,1}(t)\right)^2, & \mbox{if } m=2k-1 \\
        \left(P_k^{1,0}(t)\right)^2, & \mbox{if }m=2k
     \end{array} \right. .
\end{eqnarray}
Here $P_i^{1,1}(t)$ and $P_i^{1,0}(t)$ are polynomials called adjacent\footnote{In fact, they are (normalized) Jacobi polynomials
with parameters }
by Levenshtein 
(see \cite{lev92,Lev-chapter}). What is important for us is that $P_i^{1,1}(t)=P_i^{(n+2)}(t)$ is again a Gegenbauer polynomial, 
in particular, it is an even or odd function. 

\begin{theorem}
We have
\begin{equation} \label{lb-card}
\mathcal{M}(n,k) \geq {n+k-1 \choose k}.
\end{equation}
If a $(k,k)$-design $C \subset \mathbb{S}^{n-1}$ attains this bound, then all inner products $\langle x,y \rangle$ of distinct $x,y \in C$
are among the zeros of $P_k^{(n+2)}(t)$. 
\end{theorem}

\begin{proof}
We are going to use the polynomial $f(t)=\left(P_k^{(n+2)}(t)\right)^2=d_{2k+1}(t)/(t+1)$
in Theorem \ref{thm_lp}a).
It is obvious that $f(t) \geq 0$ for every $t \in [-1,1]$. Moreover, since $P_k^{(n+2)}(t)$ is an odd or 
even function, its square is an even function. Then $f_i=0$ for every odd $i$ in the Gegenbauer expansion of 
our $f(t)$ and we conclude that $f \in M_{n,k}$. 

The calculation of $f(1)/f_0$ follows from the classical one by noting that (obviously) $f(1)=d_{2k+1}(1)/2$ 
and the Gegenbauer coefficients $f_0$ of the polynomials $f(t)$ and $d_{2k+1}(t)$ coincide since
\[ \int_{-1}^1 f(t) (1-t^2)^{(n-3)/2} dt =  \int_{-1}^1 d_{2k+1}(t) (1-t^2)^{(n-3)/2} dt. \]
Thus our bound $f(1)/f_0$ is equal to $D(n,2k+1)$, i.e. half of the value of the  Delsarte-Goethals-Seidel bound for $(2k+1)$-designs.

If a $(k,k)$-design $C \subset \mathbb{S}^{n-1}$ attains the bound \eqref{lb-card}, then equality in \eqref{main} 
follows (for $C$ and our $f(t)$). Since $f_iM_i(C)=0$ for every $i$, the equality $|C|=f(1)/f_0$ is equivalent to 
\[ \sum_{x,y\in C, x \neq y} \left(P_k^{(n+2)} (\langle x,y\rangle)\right)^2=0, \]
whence $P_k^{(n+2)} (\langle x,y\rangle)=0$ whenever $x $ and $y$ are distinct points from $C$. 
\end{proof}

The bound \eqref{lb-card} was obtained by Waldron \cite[Exercise 6.23]{Wal-book} in different way (see also (5.10) in \cite{DS89}
which concerns the case $T=2k$). 
The linear programming interpretation is new and answers the optimality question for $T=\{2,4,\ldots,2k\}$ (see the 
optimality discussion in Section 3 in \cite{ZBBKY17}). 

\begin{theorem}\label{opt}
The bound \eqref{lb-card} is optimal in the sense that it can not be improved by using in Theorem \ref{thm_lp}a) 
a polynomial from $M_{n,k}$ of degree at most $2k$. 
\end{theorem}

\begin{proof} We use a special case of the quadrature formula in Levenshtein's Theorem 5.39 from \cite{Lev-chapter}, namely
\begin{equation} \label{quad-2k} 
f_0=\frac{f(1)+f(-1)}{D(n,2k+1)}+\sum_{i=1}^k \rho_i^{(k)} f(t_i^{1,1}), 
\end{equation}
where the weights $\rho_i^{(k)}$ are positive and $t_1^{1,1}<t_2^{1,1}<\cdots<t_k^{1,1}$ are the 
zeros of $P_k^{1,1}(t)$. The formula \eqref{quad-2k} holds true for every real polynomial of degree
at most $2k$. Defining, as in \cite{NN03}, test functions
\[ Q_j^{(n)}(k):=\frac{P_j^{(n)}(1)+P_j^{(n)}(-1)}{D(n,2k+1)}+\sum_{i=1}^k \rho_i^{(k)} P_j^{(n)}(t_i^{1,1}), \ j=1,2,\ldots, \]
one proves that the bound \eqref{lb-card} can be improved by Theorem \ref{thm_lp}a) if and only if $Q_j^{(n)}(k)<0$ 
for some $j$. It follows from \eqref{quad-2k} that $Q_j^{(n)}(k)$, $j \leq 2k$, is equal to the Gegenbauer coefficient $f_0$ of 
$P_j^{(n)}(t)$, which is, of course, 0 for $1 \leq j \leq 2k$ (in fact, $Q_j^{(n)}(k)=0$ for every odd $j$). 
Therefore $Q_j^{(n)}(k)<0$ is impossible for $1 \leq j \leq 2k$, which completes the proof.
\end{proof}

\begin{remark} 
Optimality results using test functions as above originate from \cite{BDB96}, where necessary and 
sufficient conditions for existence of improvements of 
the Levenshtein bounds were proved (see also Theorem 5.47 in \cite{Lev-chapter}). The corresponding result for the 
Delsarte-Goethals-Seidel bound was proven in \cite{NN03}. 
\end{remark}

\section{On codes attaining the bound \eqref{lb-card}}

%The existence of $(k,k)$-designs on $\mathbb{S}^{n-1}$ for all large enough cardinalities follows from a general result of
%Seymour and Zaslavsky \cite{SZ84} from 1984. Thus it is interesting to know designs with small cardinalities. 

The basic example of spherical $(k,k)$-designs comes naturally from antipodal spherical $(2k+1)$-designs.
A spherical code $C$ is called antipodal if $C=-C$. 

\begin{example} \label{ex1} Let $C \subset \mathbb{S}^{n-1}$ be an antipodal spherical $(2k+1)$-design. Consider the
spherical code $C^\prime \subset \mathbb{S}^{n-1}$ formed by the following rule: from each pair $(x,-x)$ of 
antipodal points of $C$ exactly one of the points $x$ and $-x$ belongs to $C^\prime$. Then $C^\prime$ is
a spherical $(k,k)$-design. Indded, it is easy to see in \eqref{Mk0} that $M_{2i}(C^\prime)=M_{2i}(C)/2=0$ for $i=1,2,\ldots,k$
because of $|C^\prime|=|C|/2$ and $P_{2i}^{(n)}(t)=P_{2i}^{(n)}(-t)$ for every $t$.

So any orthonormal basis is an $(1,1)$-design and its "doubling" gives a (tight) spherical 3-design. Further, any six 
points of the icosahedron no two of which are antipodal form a $(2,2)$-design since the icosahedron is a (tight) 5-design. 
There are many similar examples (see \cite{Wal-book}).
\end{example}

The other direction of Example \ref{ex1} works as follows. If $C \subset \mathbb{S}^{n-1}$ is a $(k,k)$-design 
and $C \cap -C = \phi$, then $C \cup -C$ is an antipodal $(2k+1)$-design by using \eqref{Mk0}.

It follows from Example \ref{ex1} and its reverse that the bound \eqref{lb-card} is attained exactly when there exist 
an antipodal spherical $(2k+1)$-design with $ 2{n+k-1 \choose k}$ points. Such designs are called tight and were 
classified by Bannai and Damerell \cite{BD1,BD2}. Their classification immediately implies the following. 

\begin{theorem} \label{tight-bd}
If $C \subset \mathbb{S}^{n-1}$ is a $(k,k)$-design with $\mathcal{M}(n,k)={n+k-1 \choose k}$ points, then one of the following holds true:

(i) $k=1$ and $C$ defines an orthonormal basis of $\mathbb{R}^n$;

(ii) $k=2$, $n=3$ or $n=u^2-2$, where $u$ is an odd positive integer; 

(iii) $k=3$, $n=3v^2-4$, where $v \geq 2$ is a positive integer;

(iv) $k=5$, $n=24$.  
\end{theorem}

Examples for (ii) and (iii) are only known for $u=3$ and 5 and $v=2$ and 3, respectively. The distance distributions of the 
related tight spherical 5- and 7-designs for (ii) and (iii) were found by the author in \cite{Boy95}. The related tight 11-design 
for (iv) is formed by the $2{28 \choose 5}$ vectors of minimum norm in the Leech lattice. 

\begin{theorem}
There exist no $(2,2)$-designs on $\mathbb{S}^{n-1}$, $n \geq 3$, with ${n+1 \choose 2}+1$ points. 
\end{theorem}

\begin{proof}
We first see that a spherical $(2,2)$-design of $1+n(n+1)/2$ points cannot possess a pair of antipodal points. 
Assume that $C$ is such a design. Using the Gegenbauer expansion of $t^4$ and the conditions $M_2(C)=M_4(C)=0$, 
we obtain by Theorem \ref{moments-degs} that
\[ 1+\sum_{x \in C \setminus \{y\}} \langle x,y \rangle^4 = \frac{3|C|}{n(n+2)} \]
for any fixed $y \in C$. Using this for $y$ such that $\langle x,y \rangle =-1$ for some $x \in C$, we obtain
$3|C|/n(n+2)-2 \geq 0 \iff 3(n^2+n+2) \geq 4n(n+2)$, 
which gives a contradiction. 

Let $C \subset \mathbb{S}^{n-1}$ is a $(2,2)$-design with ${n+1 \choose 2}+1$ points. Since $C \cap -C =\phi$, we conclude that
$C \cup -C$ is an antipodal 5-design with $n^2+n+2$ points. Now the proof is completed by noting that the 
nonexistence of such designs for $n \geq 3$ was shown by Reznick \cite{Rez95}. 
\end{proof}


\begin{thebibliography}{99}
\bibitem{NN03} S.\,Nikova, V.\,Nikov, Improvement of the Delsarte bound for $\tau$-designs when it is 
not the best bound possible, {\it Designs, Codes, Crypt.} 28 (2003) 201-222.
\end{thebibliography}
\end{document}